\documentclass[ssy, preprint,reqno]{imsart}
\setattribute{journal}{name}{}
\usepackage{latexsym,epsfig,amssymb,amsmath,amsfonts,amsthm,url,bbm,enumerate,float, fancyhdr}
\usepackage[usenames,dvipsnames]{color}
\usepackage[comma,sort&compress]{natbib}
\usepackage{graphicx,hyperref} 

\allowdisplaybreaks 
\numberwithin{equation}{section}
\newtheorem{Theorem}{Theorem}[section]
\newtheorem{Lemma}[Theorem]{Lemma}
\newtheorem{prop}[Theorem]{Proposition}
\newtheorem{cor}[Theorem]{Corollary}

\theoremstyle{definition}
\newtheorem{Remark}[Theorem]{Remark}
\theoremstyle{definition}

\def\E{\mathbb{E}}

\def\N{\mathbb{N}}
\def\P{\mathbb{P}}

\def\PA{\text{PA}}

\def\bone{\boldsymbol 1}

\definecolor{darkred}{RGB}{139,0,0}
\definecolor{darkgreen}{RGB}{0,139,0}

\begin{document}

\begin{frontmatter}

\title{Growth of Common Friends in a Preferential Attachment Model}
\runtitle{Common friends}
\begin{aug}
  \author{\fnms{Bikramjit}  \snm{Das}\ead[label=e1]{bikram@sutd.edu.sg}}
 \and
  \author{\fnms{Souvik}  \snm{Ghosh}\ead[label=e2]{sghosh@linkedin.com}}

 \thankstext{T1}{The authors gratefully acknowledge support from MOE Tier 2 grant MOE2017-T2-2-161.}

  \runauthor{B. Das \and S. Ghosh}

  \affiliation{Singapore University of Technology and Design\thanksmark{m1} \and LinkedIn\thanksmark{m2}}

  \address{Singapore University of Technology and Design\\20 Dover Drive, Singapore 138682 \\
           \printead{e1}}

  \address{LinkedIn Corporation, 700 E. Middlefield Road,\\ Mountain View, CA 94043, USA\\
          \printead{e2}}

\end{aug}

\begin{abstract}
The number of common friends (or connections)  in a graph is a commonly used measure of proximity between two nodes. Such measures are used in link prediction algorithms and recommendation systems in large online social networks. We obtain the rate of growth of the number of common friends in a linear preferential attachment model. We apply our result to develop an estimate for the number of common friends. We also observe  a phase transition in the limiting behavior of the number of common friends; depending on the range of the parameters of the model, the growth is either power-law, or, logarithmic, or  static with the size of  the graph.

\end{abstract}



\end{frontmatter}

\section{Introduction}
 Networks platforms like LinkedIn, Facebook, Instagram and Twitter form a big part of our culture. These networks have facilitated an increasing number of  personal as well as professional interactions. 
The networking platforms strive to grow the network (graph) both in terms of the number of users (nodes) and the number of friendships or connections (edges) since a more densely connected user network typically results in a more engaged user base.  The platforms often use recommendation systems like \emph{People You May Know} (LinkedIn, Facebook) or \emph{Who to Follow} (Twitter) \citep{Gupta:2013}, that recommend individual users to connect with other users on the platform. Such recommendation systems look for signals that indicate that two individuals might know each other. For example, having a \emph{common friend} between two users is a signal that they know each other. Furthermore, if two users have many friends in common then there is a high chance that they know each other. 
A generalization of this problem is that of \emph{link prediction} in a network and is well-studied in the literature \citep{ASI:ASI20591}.  

%

In this paper we establish the rate of growth of common friends for a fixed pair of nodes in a linear preferential attachment model, a commonly used generative graph model.  The preferential attachment model, made popular in \cite{albert:barabasi:1999}, is a very well studied class of graph models. 
Studies have covered the behavior of degree sequence \citep{bollobas_etal:2001b,resnick:samorodnitsky:towsley:davis:willis:wan:2016,resnick:samorodnitsky:2014}, the maximal degrees in a graph, second-order degree sequences (size of network of friends of friends) \citep[Section 8]{vanderHofstad:2016}, generalizations to sublinear preferential attachment \citep{dereich:morters:2009} and limiting structure of networks \citep{elwes:2016}. Other generalizations and extensions of these scale-free models have been  studied in \cite{cooper:frieze:2003,bollobas:borgs:chayes:riordan:2003}.  See \cite{vanderHofstad:2016} for a nice overview and proper definitions of the models. To the best of our knowledge this is the first theoretical study of the number of common friends in a graph.

 Two important observations follow from our result:
\begin{itemize}
\item There is  a phase transition in the asymptotic behavior of common friends. Depending on  parameter values of the preferential attachment model, the number of common friends can exhibit a power-law or logarithmic growth or be static with the growth of the graph.

\item A corollary of our result is that we can use sampling techniques to estimate the common friends in a large network. This is helpful because computing the number of common friends for every pair of nodes in a graph is computationally expensive, especially for large networks with hundreds of millions of nodes and hundreds of billions of edges.
\end{itemize}


This paper is organized as follows. In Section \ref{sec:model} we describe the \emph{linear preferential attachment model} we work with and state the main result. In Section \ref{sec:simulation} we show some simulated results providing intuition for our results. We provide the proof of the main result and some required supplementary results  in Section \ref{sec:growth}. We conclude indicating future direction of work in Section \ref{sec:conclusion}.



\section{Growth of Common Friends: Main Result} \label{sec:model}
The model paradigm we work with is a version of the well-known  \emph{undirected linear preferential attachment graph}. The idea is that at every time instance when a new node comes to the network it creates $C$ independent edges and attaches to the previous nodes following a preferential attachment rule. The process is described as follows:

 At any time $n \ge 1$, the graph sequence is denoted $\PA_{n}^{\delta,C}$ where $C\in \N^{*}=\{1,2,\ldots\}$ and $\delta>-C$. Initially, the graph $\PA_{1}^{\delta,C}$ has one node $v_{1}$ with $C$ self-loops. Then $\PA_{n}^{\delta,C}$ evolves to $\PA_{n+1}^{\delta,C}$ thus:
at the $(n+1)^{\texttt{th}}$  stage, a new node named $v_{n+1}$ is added along with $C$ edges each of which has $v_{n+1}$ as one of its vertices,
and the other vertex is selected from $V_n:=\{v_{1},v_{2},\ldots, v_{n}\}$ with probability proportional to the degree of the vertex  (shifted by a parameter $\delta$)  in $\PA_{n}^{\delta,C}$. For $1\le i \le n$:
\begin{align}\label{def:pijn}
 p_{i,n+1}:=\P[v_{i} \leftrightarrow v_{n+1}| \PA_{n}^{\delta,C}] = &  \frac{D_{i}(n)+\delta}{\sum_{j\in V_{n}} (D_{j}(n)+\delta)} = \frac{D_{i}(n)+\delta}{(2C+\delta)n}.
                                               \end{align}
Here $D_{i} (n): =$ degree of $v_{i}$ in $\PA_{n}^{\delta,C}$. The evolution of $D_{i}(n)$ occurs as:
\begin{align}\label{def:direcurrence}
D_{i}(n+1) = D_{i}(n) + \Delta_{i}(n+1)
\end{align}
where $\Delta_{i}(n+1):= $  the number of stubs of $v_{n+1}$ (out of $C$) which attaches to $v_{i}$.
Moreover at any stage $n$, for $1\le i<j\le n$, call
\[ N_{ij}(n) := \# \;\text{friends common to $i$ and $j$.}\] 
We ignore multi-edges when counting $N_{ij} (n)$ in the graph $\PA_{n}^{\delta,C}$, that is, $v_{k}$ counts as one common friend (vertex) between $v_{i}$ and $v_{j}$ in $\PA_{n}^{\delta,C}$
for $1\le i\neq j\neq k\le n$, if $v_{i}\leftrightarrow v_{k}$ and  $v_{k}\leftrightarrow v_{j}$  regardless of their multiplicity.  Our goal is to understand the behavior of $N_{ij}(n)$ for $1\le i<j\le n$ as $n$ becomes large. Observe that in our model the growth of $N_{ij}(n)$ occurs via the recurrence relation 
\begin{equation}\label{def:nijrel}
N_{ij}(n+1) = N_{{ij}} (n) + \bone_{B_{ij}(n+1)}
\end{equation}
where $B_{ij}(n+1)$ is the event $\{v_{i}\leftrightarrow v_{n+1} \leftrightarrow v_{j}\}$. Also, note that the possible range of parameters is $C\ge 2$ and $\delta>-C$. 

The power-law growth behavior for the degree distribution of a specific node in a linear preferential attachment model is well-known \cite{bollobas_etal:2001b},\cite[Section 8]{vanderHofstad:2016}. 
\begin{prop}\label{prop:degree}
For any fixed node $v_i$, we have
\begin{align*} 
\lim_{n\to\infty}\frac{D_{i}(n)}{n^{\gamma}} \to  D_{i}(\infty) \quad {\text{a.s.}},
\end{align*}
where $\gamma =\frac{C}{2C+\delta}$, and $D_{i}(\infty)$ is a non-negative random variable with \linebreak $\E [D_{i}(\infty)]=(C+\delta) \frac{\Gamma(i)}{\Gamma(i+\gamma)}$.
\end{prop}

Proposition \ref{prop:degree} can be derived using arguments from \cite[Proposition 8.2]{vanderHofstad:2016}  or  using similar arguments as in the proof of the  Proposition \ref{prop:degdeg} provided in Section \ref{sec:growth}.

Our main contribution is the following theorem which provides the growth rate of number of common friends of two nodes in such a model. The proof is given in Section \ref{sec:growth}.

\begin{Theorem}\label{thm:growingfriends}
Under the linear preferential attachment model, $(\PA_{n}^{\delta,C})_{n\ge 1}$, with $C\ge 2$, for any two fixed nodes $v_{i}, v_{j}$, we have
\begin{align*}
& \text{(1) } \lim_{n\to \infty} N_{ij} (n) = N_{ij}(\infty) \quad{\text{a.s.}} & \text{if \ } \delta >0,&\\
& \text{(2) }   \lim_{n\to \infty} \frac{N_{ij} (n)}{\log n} = \frac{C(C-1)}{(2C+\delta)^{2}}{D_{i}(\infty)D_j(\infty)}\quad {\text{a.s.}} & \text{if \ } \delta =0,&\\
& \text{(3) }   \lim_{n\to \infty}  \frac{N_{ij} (n)}{n^{2\gamma-1}/(2\gamma-1)} = \frac{C(C-1)}{(2C+\delta)^{2}}{D_{i}(\infty)D_j(\infty)} \quad {\text{a.s.}}  & \text{if \ } \delta <0.&
\end{align*}
where $\gamma =\frac{C}{2C+\delta}$, $\gamma_{1}=(1-\frac1{\sqrt{C}})\gamma$, $\gamma_{2}=(1+\frac1{\sqrt{C}})\gamma$. Furthermore, $\E [N_{ij}(\infty)]  < \infty$ and $Y_{ij}(\infty) = D_i(\infty)D_j(\infty)$ with $\E [Y_{ij}(\infty)]  =(C+\delta)^{2} \frac{ \Gamma(i)\Gamma(j)\Gamma(j+\gamma)}{\Gamma(i+\gamma)\Gamma(j+\gamma_{1})\Gamma(j+\gamma_{2})}$; $D_i(\infty)$ is the limit of the scaled degree sequence of node $v_i$ as defined in Proposition \ref{prop:degree}.
\end{Theorem}

\begin{Remark}
An interesting observation is the different regimes in the growth rate  of number of common friends depending on the parameter $\delta$.
\begin{enumerate}
\item When $\delta>0$, we are in a regime that is mildly preferential attachment. In this regime, the nodes with low degree also get enough number of new friends. As $\delta$ increases, more nodes have a similar chance of being selected. Although the individual degrees for a fixed node grows like a power-law behavior, the number of common friends between two fixed nodes has a finite expectation even in the limit.\\
\item For $\delta\le 0 $ and especially $\delta$ closer to $-C$, the new nodes prefer to friend nodes with a high degree. In this case the number of common friends tend to grow with the number of nodes,  as a power-law for $\delta<0$ and at a logarithmic rate for $\delta=0$.
\end{enumerate}
\end{Remark}

\begin{cor} \label{cor1}
Under the preferential attachment model, $(\PA_{n}^{\delta,C})_{n\ge 1}$, with $C\ge 2$, for any two fixed nodes $v_{i}, v_{j}$ and $k>1$, we have
\begin{align*}
& \text{(1) }    \lim_{n\to \infty} \frac{N_{ij} (n)}{N_{ij}(\lfloor n/k\rfloor)} =1  \quad{\text{a.s.}} & \text{if \ } \delta >0 \text{\ and \ } \lim_{n\to \infty} N_{ij} (n) >0,\\
& \text{(2) }   \lim_{n\to \infty} \frac{N_{ij} (n)}{N_{ij}(\lfloor n/k\rfloor)} = 1 \quad {\text{a.s.}} & \text{if \ } \delta =0 \text{\ and \ }  Y_{ij}(\infty)>0,\\
& \text{(3) }   \lim_{n\to \infty}  \frac{N_{ij} (n)}{N_{ij}(\lfloor n/k\rfloor)k^{2\gamma-1}} = 1  \quad {\text{a.s.}}  & \text{if \ } \delta <0 \text{\ and \ } Y_{ij}(\infty)>0.
\end{align*}
\end{cor}
\begin{proof} The result is an easy application of the almost sure convergences of $N_{ij}(n)$ observed  in Theorem \ref{thm:growingfriends}.
\end{proof}
The above corollary states that we can consistently estimate the number of common friends for a given pair of nodes using an earlier state of the graph, i.e., for any $k>1$
\[ \hat N^k_{ij}(n) = \left\{\begin{array}{ll} N_{ij}(\lfloor n/k\rfloor) & \mbox{if } \delta \ge 0\\ N_{ij}(\lfloor n/k\rfloor)k^{2\gamma -1} & \mbox{if } \delta <0. \end{array} \right.\]
For a large value of $k$, the graph  $\PA_{n/k}^{\delta,C}$ can be significantly smaller than $\PA_{n}^{\delta,C}$ and it is  significantly cheaper to estimate the number of common friends.

\section{Simulation Study} \label{sec:simulation}

We  illustrate the key idea behind Theorem \ref{thm:growingfriends} in the following simulated examples. Figure \ref{figure:1} shows instances of graphs simulated from the preferential attachment model with 20 nodes and $C=2$; left one with $\delta = -1.5$ and the right one with $\delta = 1.5$. When $\delta = -1.5$ we observe that the graph grows quite  preferentially. New nodes tend to connect with the same few nodes and hence the number of common friends for them keep growing fast. When $\delta = 1.5$, the graph is more distributed. New nodes tend to connect with different nodes and hence the number of common friends does not grow so much.

\begin{figure}[H]
\centering

\begin{tabular}{cc}
  \includegraphics[width=50mm]{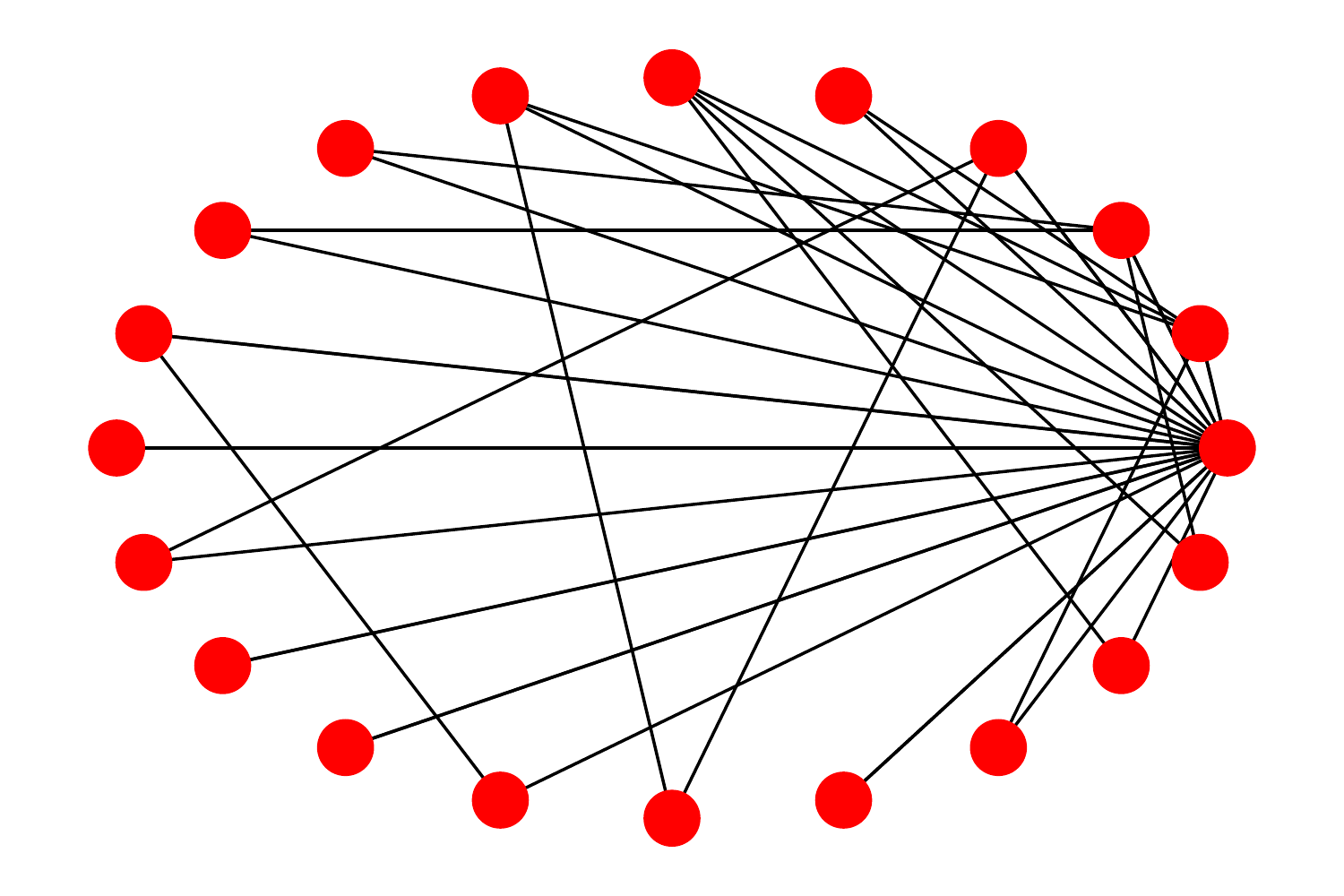} & 
  \includegraphics[width=50mm]{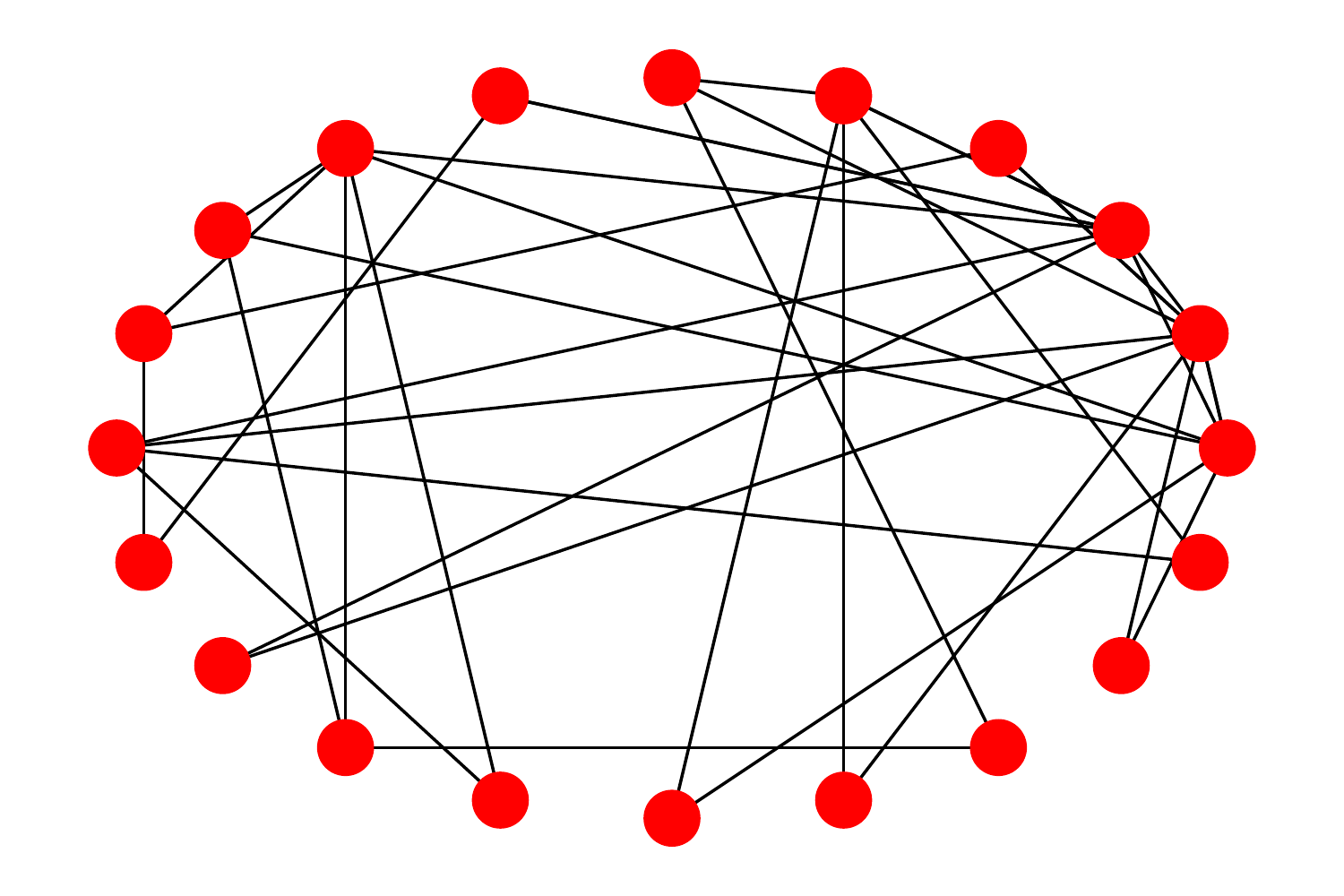} \
 
\end{tabular}

\caption{Graphs with 20 nodes simulated from $\PA_{20}^{\delta,2}$. On the left we have $\delta = -1.5$, and the right one has $\delta=1.5$.}
 \label{figure:1}
 \end{figure}
 
 \begin{figure}[H]
\centering

\begin{tabular}{cc}
  \includegraphics[width=55mm]{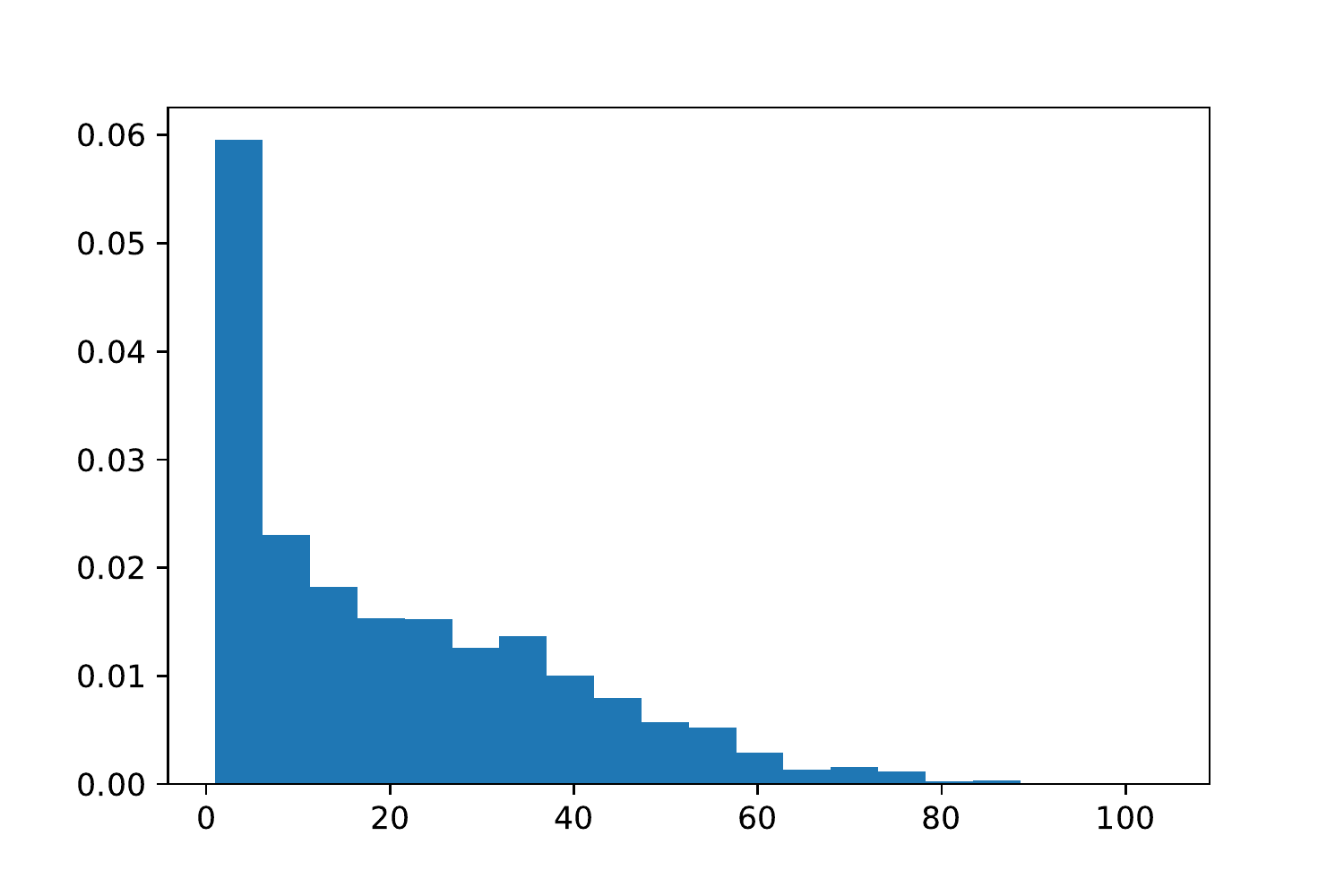} & 
  \includegraphics[width=55mm]{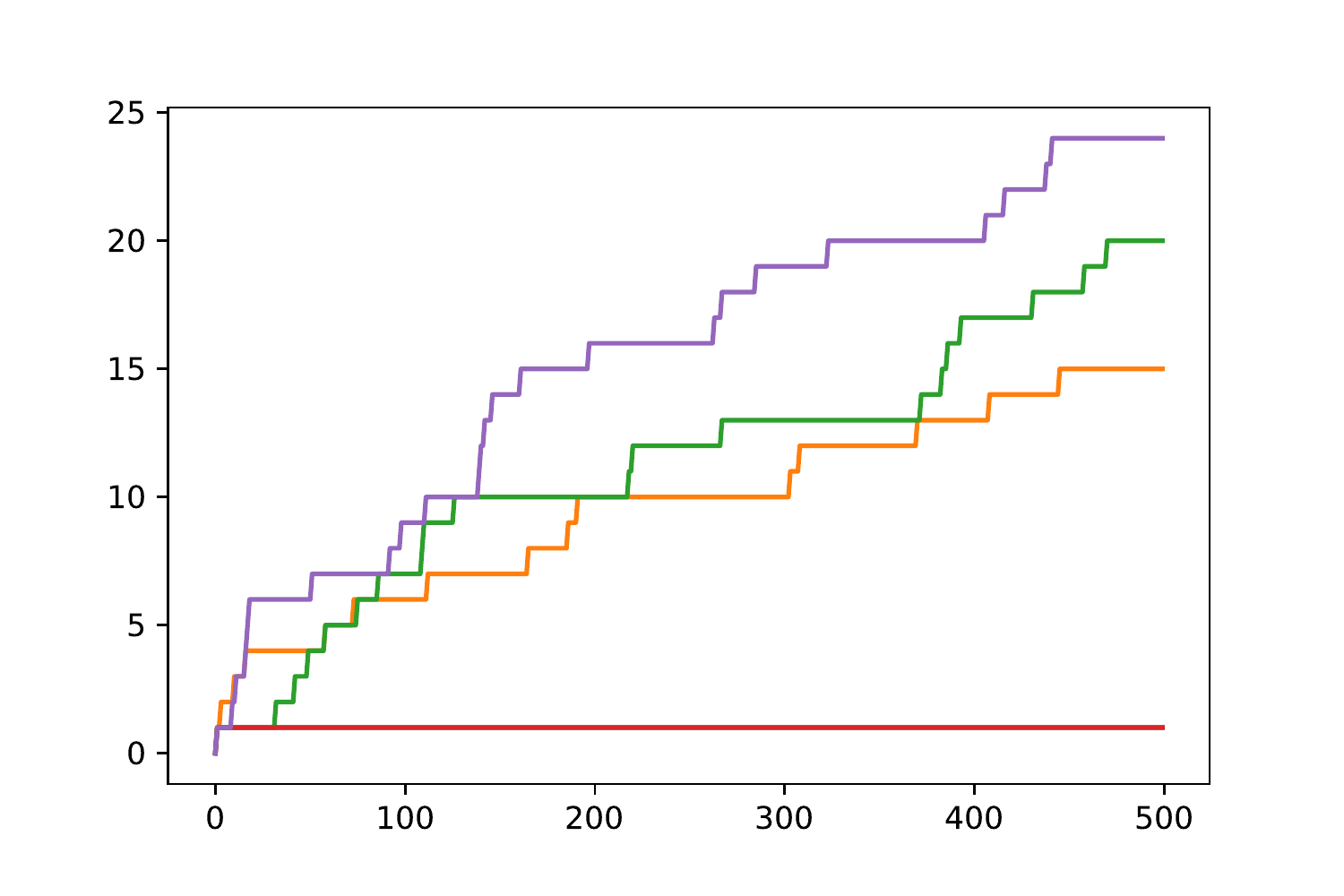} \
  
\end{tabular}

\caption{Simulation results for the number of common friends for two fixed nodes for $\PA^{-1.5,2}_{500} (\delta = -1.5, C = 2, n = 500)$. The left plot is the histogram obtained after 2500 simulations. The right plot is the growth path for 5 arbitrarily chosen simulations. }
 \label{figure:2}
 \end{figure}

\begin{figure}[H]
\centering

\begin{tabular}{cc}
  \includegraphics[width=55mm]{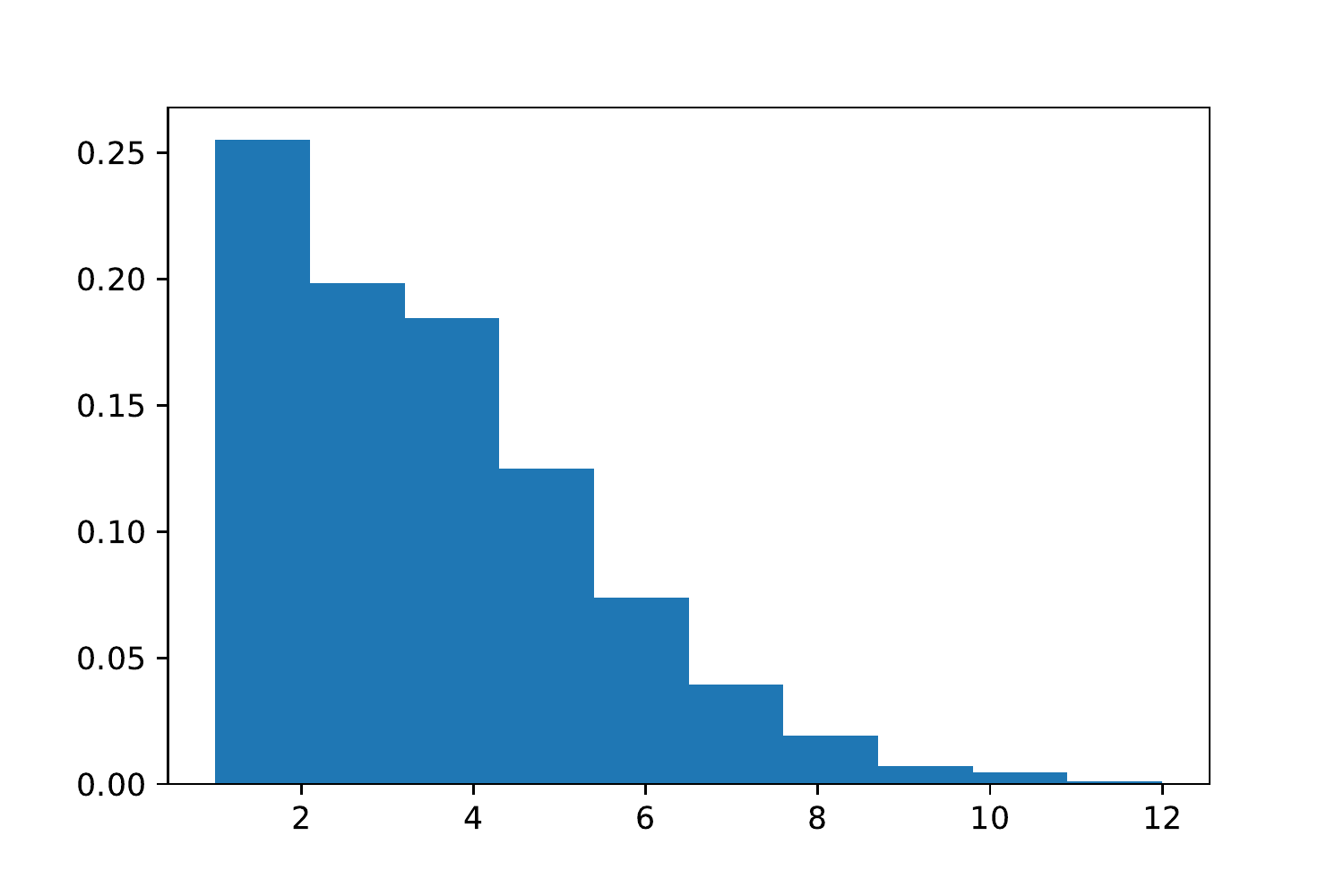} & 
  \includegraphics[width=55mm]{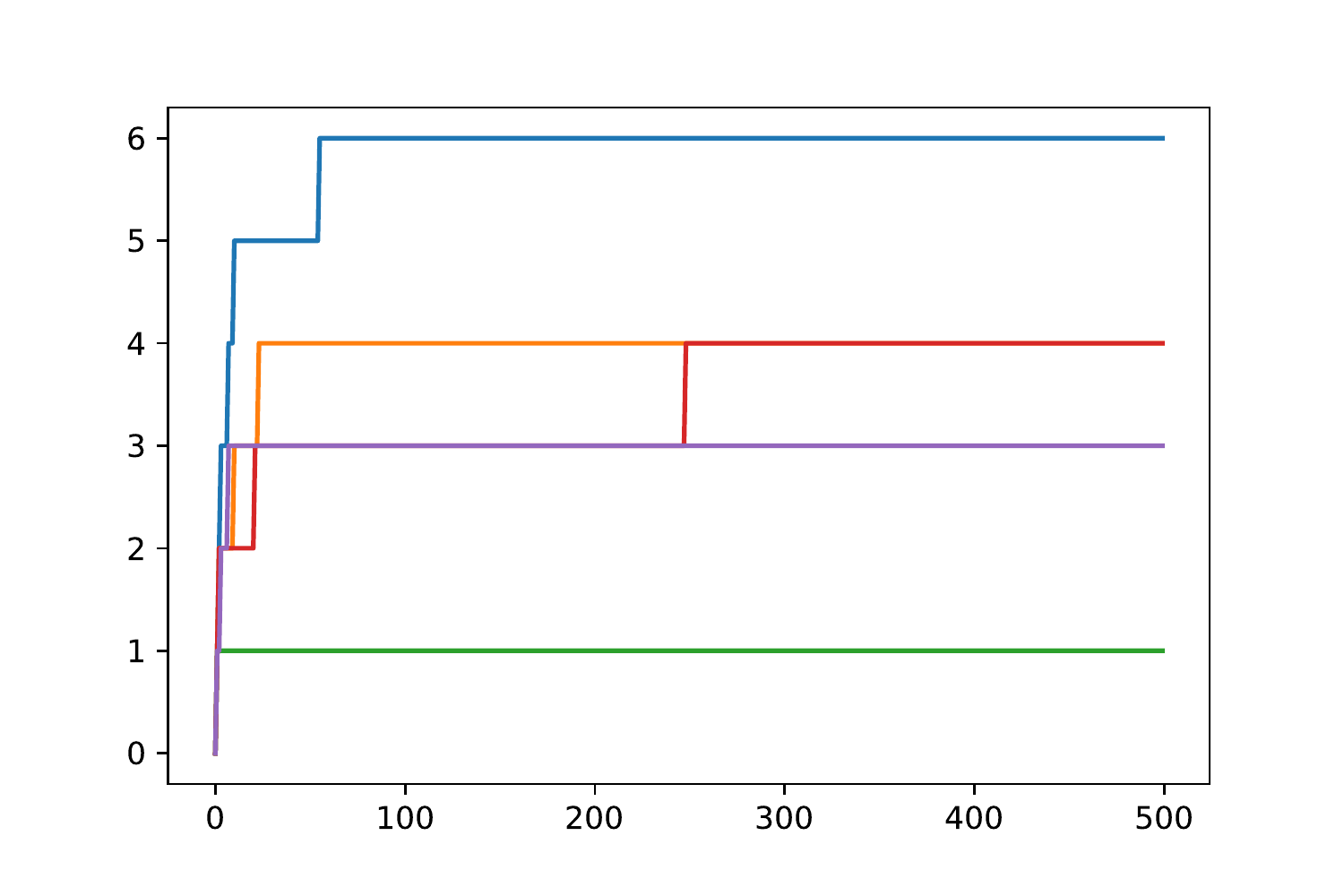} \
  
\end{tabular}

\caption{Simulation results for the number of common friends for two fixed nodes for $\PA^{1.5,2}_{500} (\delta = 1.5, C = 2, n = 500)$. The left plot is the histogram obtained after 2500 simulations. The right plot is the growth path for 5 arbitrarily chosen simulations.}
 \label{figure:3}
 \end{figure}

To understand asymptotic property of the behavior of common friends, we simulate larger graphs and replicate the exercise multiple times.  The left plot in Figure \ref{figure:2} is the histogram of number of common friends for two fixed nodes ($v_{10}$ and $v_{20}$) for $\PA^{-1.5,2}_{500}$ replicated 2500 times shows the heavy-tailed phenomenon. We also show trajectories of the number common friends for 5 arbitrary simulations as the size of the network grows in the right plot of Figure \ref{figure:2}. 

 Figure \ref{figure:3} shows the  behavior of common friends when $\delta =1.5$. As expected, we see that common friends do not grow that fast in this case.

\begin{figure}[h]
\centering

\includegraphics[width=12cm]{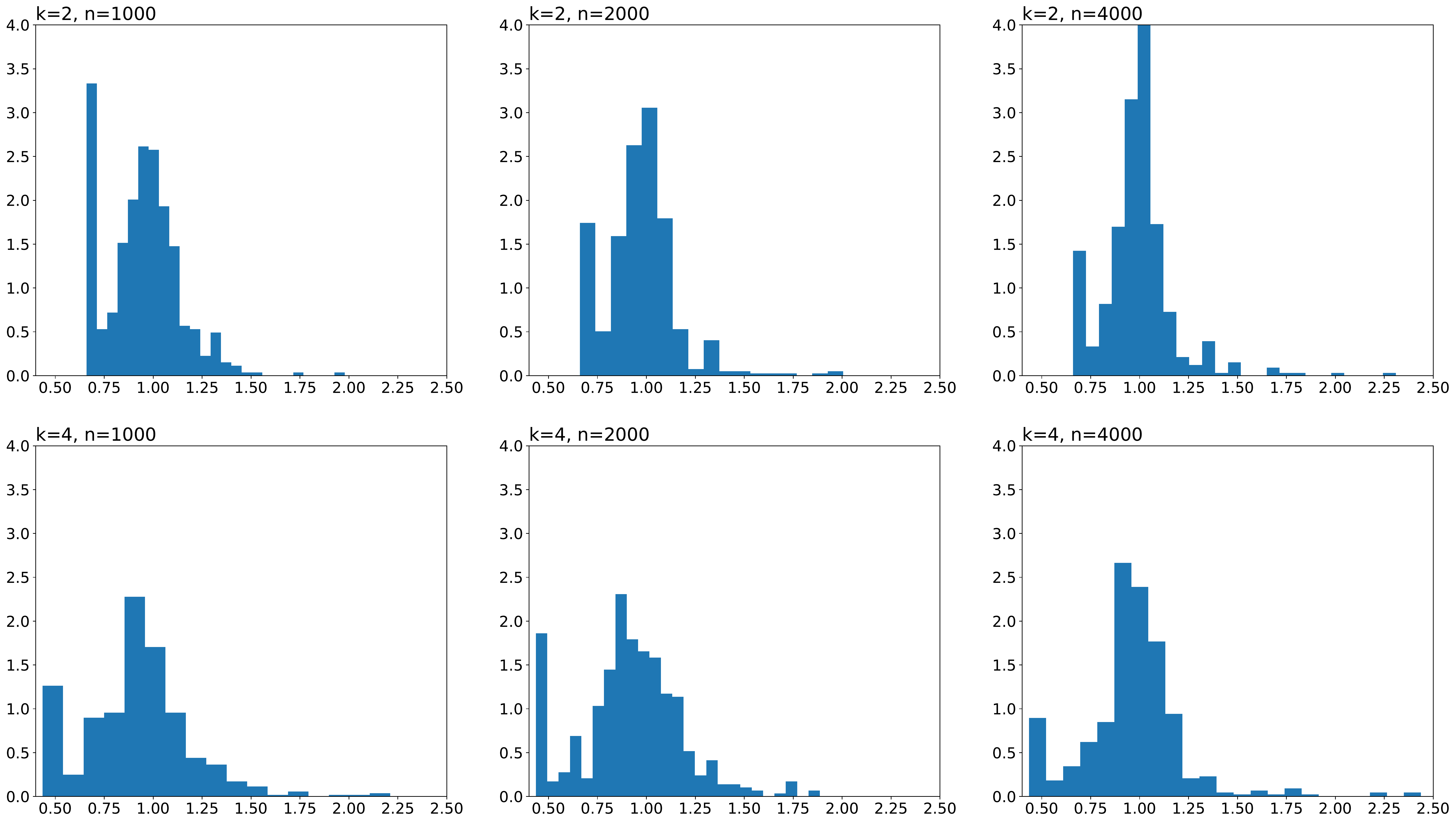} 
  
\caption{Simulation results for estimating the number of common friends for a $\PA^{1.5,2}_{n}$ model. The first row shows the histogram of $N_{ij}(n)/\hat N^k_{ij}(n)$ for $k=2$ and $n=1000, 2000, 4000$. The second row shows the same for $k=4$. }
 \label{figure:4}
 \end{figure}
 
 We also check the validity of Corollary \ref{cor1} using simulations.
Figure \ref{figure:4} provides simulation results for the estimator $\hat N^k_{ij}(n)$ for a $\PA^{1.5,2}_{n}$ model. The first row shows the histogram of 500 simulations of $N_{ij}(n)/\hat N^k_{ij}(n)$ for $k=2$ and $n=1000, 2000, 4000$. The second row shows the same for $k=4$. We see the concentration of $N_{ij}(n)/\hat N^k_{ij}(n)$ near 1 as $n$ increases.

\section{Proof of Main Result} \label{sec:growth}
In this section, we prove Theorem \ref{thm:growingfriends} by observing the asymptotic behavior of $D_{i}(n)$, $(D_{i}(n), D_{j}(n))$ jointly and functions thereof. Recall that our preferential attachment graph sequence is $(\PA_{n}^{\delta, C})_{n\ge1}$. {We assume that $C\ge 2$  and $\delta>-C$ for all the results in this section.}  For convenience's sake we use the following notations:
\begin{align*}
X_{i}(n) &:= D_{i}(n) + \delta,\\
Y_{ij}(n) &:= (D_{i}(n) + \delta)(D_{j}(n) + \delta) = X_{i}(n)X_{j}(n).
\end{align*}
From Proposition \ref{prop:degree} we get 
\begin{align} \label{eq:x:conv}
\lim_{n\to\infty}\frac{X_{i}(n)}{n^{\gamma}} \to  D_{i}(\infty) \quad {\text{a.s.}},
\end{align}
where $\gamma =\frac{C}{2C+\delta}$.

In the next few steps we apply Martingale Convergence Theorem to show almost sure convergence for the appropriately scaled sequences of random variables $\{Y_{ij}(n)\}$. We also prove uniform integrability of the sequences $\{X_i(n)\}$ and $\{Y_{ij}(n)\}$ so that we additionally have convergence in $\mathcal{L}_{1}$ and can compute the expectation of the limit.
\begin{Lemma}\label{lem:XisUI}
For any $k\ge 1$ and for a fixed $i$, we have
 \[\sup_{n\ge i} \frac{\E[X_{i}(n)^{k}]}{n^{k\gamma}} <\infty.\]
\end{Lemma}
\begin{proof}
 We prove the result by induction on $k$. We have $ X_{i}(n+1) = X_{i}(n)  +\Delta_{i}(n+1)$ and $\Delta_{i}(n+1)|\PA_{n}^{\delta,C} \sim \text{Binomial} (C,p_{i,n+1})$. Hence  for $k=1$ with $1\le i\le n$ we get
\begin{align*}
\E[X_{i}(n+1)|\PA_{n}^{\delta,C}] & = X_{i}(n)  + \E[\Delta_{i}(n+1)|\PA_{n}^{\delta,C}] = X_{i}(n) + Cp_{i,n+1}\\
                                               & = (D_{i}(n)+\delta) + C\cdot\frac{D_{i}(n)+\delta}{(2C+\delta)n} = X_{i}(n) \left(1+\frac{\gamma}{n}\right).
                                               \end{align*}
We also have $ \E[X_{i}(i)] = D_{i}(i)+\delta= C +\delta.$ Now
 \begin{align*}
\E[X_{i}(n+1)] & = \E[X_{i}(n)]\left(\frac{n+\gamma}{n}\right)\\ 
& =\E[X_{i}(i)]\left(\frac{i+\gamma}{i}\right)\left(\frac{i+1+\gamma}{i+1}\right)\cdots\left(\frac{n+\gamma}{n}\right)\\
				 & = (C+\delta) \frac{\Gamma(i)}{\Gamma(i+\gamma)}\frac{\Gamma(n+1+\gamma)}{\Gamma(n+1)} =(C+\delta) \frac{\Gamma(i)}{\Gamma(i+\gamma)} (n+1)^{\gamma} \left(1+O\left(\frac1n\right)\right)
\end{align*}
using Stirling's formula \citep{abramowitz:stegun:2012} given by
\begin{align}\label{eqn:stirling} 
\frac{\Gamma(n+a)}{\Gamma(n)} = n^{a} \left(1+O\left(\frac1n\right)\right).
\end{align}
Hence for any $n\ge i$,
\[ \frac{\E[X_{i}(n)]}{n^{\gamma}} = (C+\delta) \frac{\Gamma(i)}{\Gamma(i+\gamma)}\left(1+O\left(\frac1n\right)\right) <\infty.\]
 Thus the result holds for $k=1$. By induction hypothesis, let the result be true for  $j=1,\ldots, k-1$ and  we have constants $C_{j}$ such that 
 \begin{align}\label{eq:boundCj}
\sup_{n\ge i} \frac{\E[X_{i}(n)^{j}]}{n^{j\gamma}} \le C_{j} <\infty.
\end{align}
Denoting $p=p_{i,n+1}=\frac{X_{i}(n)}{(2C+\delta)n}$ we get
\begin{align}
& \E\left[X_{i}(n+1)^{k}\right] \nonumber \\
                            & = \E\left[\E\left[ [X_{i}(n)+\Delta_{i}(n+1)]^{k}\Big|\PA_{n}^{\delta,C}\right]\right] \nonumber\\
                            & = \E\left[\E\left[X_{i}(n)^{k}+kX_{i}(n)^{k-1}\Delta_{i}(n+1) + \binom{k}{2} X_{i}(n)^{k-2}\Delta_{i}(n+1)^{2} + \cdots\Big|\PA_{n}^{\delta,C}\right]\right] \nonumber\\
                            & = \E\left[X_{i}(n)^{k}+kX_{i}(n)^{k-1}Cp + \binom{k}{2} X_{i}(n)^{k-2}(Cp(1-p)+(Cp)^{2}) + \cdots \right] \nonumber\\
                            & = \E\left[X_{i}(n)^{k}+X_{i}(n)^{k}\frac{k\gamma}{n} + X_{i}(n)^{k}\frac{k(k-1)(C-1)}{2C} \left(\frac{\gamma}{n}\right)^{2}+ X_{i}(n)^{k-1}\frac{k(k-1)}{2}\frac{\gamma}{n} +\cdots \right] \nonumber\\
                            & = \E\left[X_{i}(n)^{k}\right]\left(1+\frac{k\gamma}{n} + O\left(\frac{1}{n^{2}}\right)\right) + \sum_{r=1}^{k-1} \E[X_{i}(n)^{k-r}] \left(\alpha_{k-r} \frac{\gamma}{n} + O\left(\frac1{n^{2}}\right)\right) \nonumber\\
                            & \le  \E\left[X_{i}(n)^{k}\right]\left(1+\frac{k\gamma}{n} + O\left(\frac{1}{n^{2}}\right)\right) + \sum_{r=1}^{k-1} C_{k-r}n^{(k-r)\gamma}\left(\alpha_{k-r} \frac{\gamma}{n} + O\left(\frac1{n^{2}}\right)\right) \quad\text{(using \eqref{eq:boundCj})} \nonumber\\
                            &  \le \E\left[X_{i}(n)^{k}\right]\left(1+\frac{k\gamma}{n} + O\left(\frac{1}{n^{2}}\right)\right) + C^{*} n^{(k-1)\gamma-1} \nonumber\\
                            & =: a_{n} \E\left[X_{i}(n)^{k}\right] + b_{n} \label{eq:anbnrec}
                            \end{align}
where $\alpha_{i}'s$ and $C^{*}$ (appropriately chosen) are constants, and we denote \linebreak $a_{n}=1+\frac{k\gamma}{n}+O\left(\frac{1}{n^{2}}\right)$, $b_{n}= C^{*} n^{(k-1)\gamma-1}$.    Now using \eqref{eq:anbnrec} recursively we get
 \begin{align*}
\E\left[X_{i}(n+1)^{k} \right] & \le   \E\left[X_{i}(i)^{k}\right] \prod_{r=i}^n a_{r} + \sum_{\ell=i}^{n} b_{\ell} \left(\prod_{r=\ell+1}^n a_{\ell} \right)
 \end{align*}   
   Using Sterling's formula we have for any $\ell<n$,
   \begin{align*}
    \prod_{r=\ell+1}^n a_{r}  & =  \frac{n^{k\gamma}}{a_1\ldots a_\ell} \left(1+O\left(\frac1{n}\right)\right) \quad (n\to\infty),\\
   & = \frac{n^{k\gamma}}{\ell^{k\gamma}} \left(1+O\left(\frac1{\ell}\right)\right) \quad\quad(n>\ell\to\infty).
    \end{align*}
   Therefore,
\begin{align*}
\E\left[X_{i}(n+1)^{k} \right]    & \le \frac{(C+\delta)^{k}}{a_{1}\ldots a_{i-1}} \left[n^{k\gamma} \left(1+O\left(\frac1n\right)\right)\right] +  \sum_{\ell=i}^{n} C^{*} \ell^{(k-1)\gamma +1}\frac{n^{k\gamma}}{\ell^{k\gamma}} \left(1+O\left(\frac1\ell \right)\right)\\
& \le A_1n^{k\gamma}\left(1+O\left(\frac1n\right)\right) + C^* n^{k\gamma} \sum_{\ell=i}^{n} \frac{1}{\ell^{1+\gamma}} \left(1+O\left(\frac1\ell \right)\right)
                                 \end{align*}   
where $A_1= {(C+\delta)^{k}}/(a_{1}\ldots a_{i-1})$. Hence dividing both sides by $(n+1)^{k\gamma}$ we get
         \begin{align*}
\frac{\E\left[X_{i}(n+1)^{k}\right]}{(n+1)^{\gamma k}}  & \le A_{1}\left(1+O\left(\frac1n\right)\right) + C^* \sum_{\ell=i}^{n} \frac{1}{\ell^{1+\gamma}}\left(1+O\left(\frac1\ell \right)\right) <\infty, \quad \text{for any } n\ge i.
                                 \end{align*} 
Hence the result holds for $j=k$.
\end{proof}

\begin{Lemma}\label{lem:YijisUI}
For any $k\ge 1$, we have for a fixed $i<j$,
 \[\sup_{n\ge j} \frac{\E[Y_{ij}(n)^{2k}]}{n^{2k\gamma}} <\infty.\]
\end{Lemma}

\begin{proof}
This follows from Lemma \ref{lem:XisUI} and the Cauchy-Schwarz inequality.
\end{proof}

\begin{Remark}\label{rem:UI}
Since both sequences $\left\{{X_{i}(n)}/{n^{\gamma}}\right\}_{n\ge i}$ and $\left\{{Y_{ij}(n)}/{n^{\gamma}}\right\}_{n\ge j}$ are  $\mathcal{L}^{k}$ bounded for some $k>1$ by Lemmas \ref{lem:XisUI} and \ref{lem:YijisUI}, they are also  uniformly integrable; see \cite[Theorem 4.6.2]{durrett:2019}. 
\end{Remark}



The next Proposition \ref{prop:degdeg} describes the asymptotic behavior of product of the degrees of two nodes, which as expected  also has a power-law growth. 
\begin{prop}\label{prop:degdeg}
 For any $i<j$ we have
\begin{align*}
\lim_{n\to\infty}\frac{Y_{ij}(n)}{n^{2\gamma}} \to Y_{ij}(\infty) =D_i(\infty)D_j(\infty)\quad {\text{a.s.}},
\end{align*}
where  $\E [Y_{ij}(\infty)] = (C+\delta)^{2} \frac{ \Gamma(i)\Gamma(j)\Gamma(j+\gamma)}{\Gamma(i+\gamma)\Gamma(j+\gamma_{1})\Gamma(j+\gamma_{2})}=:C_{ij}$ with $\gamma =\frac{C}{2C+\delta}$, $\gamma_{1}=(1-\frac1{\sqrt{C}})\gamma$, $\gamma_{2}=(1+\frac1{\sqrt{C}})\gamma$. Here $D_i(\infty), D_j(\infty)$ are as defined in Proposition \ref{prop:degree}. \end{prop}

\begin{proof}

Note that $$p_{i,n+1}p_{j,n+1} = \frac{Y_{ij}(n)}{(2C+\delta)^{2}n^{2}}$$
and writing $p_i=p_{i,n+1}, p_j=p_{j,n+1} $ we have
\begin{align*}
&\E [Y_{ij}(n+1)| \PA_{n}^{\delta,C}] \\
& = Y_{ij} (n) + \E [Y_{ij}(n+1)-Y_{ij}(n)| \PA_{n}^{\delta,C}]\\    
& = Y_{ij} (n) + \sum_{\substack{0\le k, l \le C\\ k + l =C}} \Big[k(D_{j}(n)+\delta)    + l(D_{i}(n)+\delta) + k l\Big] \frac{C!\; p_{i}^{k}p_{j}^{l}(1-p_{i}-p_{j})^{C-k-l}}{k!l!(C-k-l!)}\\
& = Y_{ij}(n) \frac{(n+\gamma_{1})(n+\gamma_{2})}{n^{2}},
\end{align*}
where $\gamma =\frac{C}{2C+\delta}, \gamma_{1}=(1-\frac1{\sqrt{C}})\gamma, \gamma_{2}=(1+\frac1{\sqrt{C}})\gamma$. Moreover, for $i<j$ we have,
\begin{align*}
\E Y_{ij} (j) = \E [(D_{i}(j)+\delta)(D_{j}(j)+\delta)] = (C+\delta)^{2} \frac{\Gamma(i)\Gamma(j+\gamma)}{\Gamma(j)\Gamma(i+\gamma)}.
\end{align*}
Therefore, 
\begin{align}
\E Y_{ij} (n+1) & = \E Y_{ij} (n) \frac{(n+\gamma_{1})(n+\gamma_{2})}{n^{2}} \nonumber \\
                       & = \E Y_{ij} (j)  \frac{(j+\gamma_{1})(j+\gamma_{2})}{j^{2}} \ldots \frac{(n+\gamma_{1})(n+\gamma_{2})}{n^{2}} \nonumber\\
                       & = (C+\delta)^{2} \frac{\Gamma(i)\Gamma(j+\gamma)}{\Gamma(j)\Gamma(i+\gamma)} \frac{(\Gamma(j))^{2}} {(\Gamma(n+1))^{2}} \frac{\Gamma(n+1+\gamma_{1})\Gamma(n+1+\gamma_{2})}{\Gamma(j+\gamma_{1})\Gamma(j+\gamma_{2})} \nonumber \\
                       & =: C_{ij} \frac{\Gamma(n+1+\gamma_{1})\Gamma(n+1+\gamma_{2})}{(\Gamma(n+1))^{2}}. \label{eq:yijnp1}
\end{align}
Define $$W_{ij}(n):=  \frac{Y_{ij}(n)}{C_{ij}}  \frac{(\Gamma(n))^{2}}{\Gamma(n+\gamma_{1})\Gamma(n+\gamma_{2})} = \frac{Y_{ij}(n)}{C_{ij}n^{2\gamma}}\left(1+O\left(\frac 1n\right)\right),$$
using Sterling's formula. Hence by Lemma \ref{lem:YijisUI}, $\{W_{ij}(n)\}_{j\ge n}$ is uniformly integrable. 
Moreover $W_{ij}(n)\ge 0, \E W_{ij}(n) =1,$ and $\E[W_{ij}(n+1)|W_{ij}(n)] = W_{ij}(n)$ for $1\le i\neq j \le n$. Hence by  Doob's Martinagale Convergence Theorem  \citep[Theorem 4.2.11 and Theorem 4.6.4]{durrett:2019} \[ \lim_{n\to \infty} W_{ij}(n) \to W_{ij}(\infty), \quad \text{a.s. and in $\mathcal{L}^{1}$},\]
where $W_{ij}(\infty):=\limsup_{n} W_{ij}(n)$ and $\E W_{ij}(\infty) <\infty$. Hence we have \begin{align*}
\lim_{n\to\infty}\frac{Y_{ij}(n)}{n^{2\gamma}} = C_{ij}\lim_{n\to\infty}{W_{ij}(n)}\left(1+O\left(\frac 1n\right)\right) \to C_{ij}W_{ij}(\infty)=: Y_{ij}(\infty)
\end{align*}
both almost surely and in $\mathcal{L}^{1}$ with $\E(Y_{ij}(\infty))=C_{ij}$. From Proposition \ref{prop:degree} and \eqref{eq:x:conv} we can check that $n\to \infty$, 
${Y_{ij}(n)}/{n^{2\gamma}} \to D_{i}(\infty)D_{j}(\infty)$ a.s.; hence $ Y_{ij}(\infty)=D_{i}(\infty)D_{j}(\infty)$ a.s.
\end{proof}

\begin{Lemma}\label{lem:sumright}
For any $i<j$, we have
\begin{align*}
(1)& \lim_{n\to \infty} \frac{1}{\log(n)} \sum_{k=j}^{n} \frac{Y_{ij}(k)}{k^{2}} = Y_{ij} (\infty) \quad \quad \text{a.s} & \text{if \ } \delta =0,&\\
(2) &\lim_{n\to \infty} \frac{1}{n^{2\gamma-1}/(2\gamma-1)} \sum_{k=j}^{n} \frac{Y_{ij}(k)}{k^{2}} = Y_{ij} (\infty) \quad \quad \text{a.s} & \text{if \ } \delta <0,&\\
\end{align*}
where $Y_{ij}(\infty)$ is as defined in Proposition \ref{prop:degdeg}.
\end{Lemma}

\begin{proof}
Let all our random variables be defined on the probability space $(\Omega,\mathcal{A},\P)$. From Proposition \ref{prop:degdeg}, for $\omega \in \Omega$,
\begin{align}\label{eq:yijn}
 \lim_{n\to\infty} \frac{Y_{ij}(n,\omega)}{n^{2\gamma}} = Y_{ij}(\infty, \omega)
 \end{align}
 holds with probability 1. Fix such an $\omega\in \Omega$.
Then given any small $\epsilon>0$, there exists $n_{0}\in \N^{*}$ such that for any $n\ge n_{0}$,
\begin{align}\label{eq:eps}
\Bigg|\frac{Y_{ij}(n,\omega)}{n^{2\gamma}} - Y_{ij}(\infty,\omega)\Bigg| <\epsilon.
\end{align}
(1) If $\delta=0$, we have $\gamma=\frac12$.  Also for $n\to \infty$, we have $\displaystyle{\sum_{k=1}^{n} \frac1{k^{2-2\gamma}} =\sum_{k=1}^{n} \frac1k \sim \log(n)}$.  Hence
\begin{align*}
 \Bigg|\ \frac{1}{\log(n)} & \sum_{k=j}^{n} \frac{Y_{ij}(k,\omega)}{k^{2}}- Y_{ij}(\infty,\omega)\Bigg| \\ & \le \Bigg| \frac{1}{\log(n)} \sum_{k=j}^{n_{0}} \frac{Y_{ij}(k)}{k^{2}}\Bigg| \;+ \;\Bigg|\ \frac{1}{\log(n)} \sum_{k=n_{0}+1}^{n} \left[\frac{Y_{ij}(k,\omega)}{k^{2}}- \frac{Y_{ij}(\infty,\omega)}{k}\right]\Bigg| \; \\ &\quad\quad\quad\quad\quad\quad\quad\quad\quad\quad\quad\quad+ \;\frac{Y_{ij}(\infty,\omega)}{\log(n)}\left[\sum_{k=n_{0}+1}^{n}\frac1k-\log(n)\right]\\
& =: {\rm{I}}_{n}+{\rm{II}}_{n}+{\rm{III}}_{n}.
\end{align*}
Check that as $n\to\infty$, ${\rm{I}}_{n}\to0, {\rm{III}}_{n}\to0$. Since $\gamma=\frac 12$ using \eqref{eq:eps} we get
\begin{align*}
 {\rm{II}}_{n} = \Bigg|\ \frac{1}{\log(n)} \sum_{k=n_{0}+1}^{n} \frac1k\left[\frac{Y_{ij}(k,\omega)}{k^{2\gamma}}- {Y_{ij}(\infty,\omega)}\right]\Bigg|  \le \epsilon \frac{1}{\log(n)} \sum_{k=n_{0}+1}^{n} \frac1k \to \epsilon \;\; (\text{as } n\to \infty).
\end{align*}
Therefore we have, 
\begin{align}\label{eq:yijnconv} 
\frac{1}{\log(n)} \sum_{k=j}^{n} \frac{Y_{ij}(k,\omega)}{k^{2}} \to Y_{ij}(\infty,\omega).
\end{align}
By  Proposition \ref{prop:degdeg}, \eqref{eq:yijn} holds almost surely implying  \eqref{eq:yijnconv} holds almost surely and hence Lemma \ref{lem:sumright}(1) holds.

\noindent (2) For $\delta<0$, we get $\frac12<\gamma=\frac{C}{2C+\delta}<1$ which means $0<2-2\gamma<1$. Note  that for any $0<\alpha<1$, we have  $$\displaystyle{\sum_{k=1}^{n} \frac1{k^{\alpha}} \sim \frac{1}{1-\alpha}n^{1-\alpha}}  \quad \quad (n\to \infty).$$ Now  we can prove Lemma \ref{lem:sumright}(2) in the same manner as we proved (1).
\end{proof}

\begin{Lemma}\label{lem:sumleft}
For any $i<j$, we have
\begin{align*}
(1)& \lim_{n\to \infty} \frac{1}{\log(n)} \sum_{k=j}^{n} \frac{Y_{ij}(k)}{k^{2}} \left(1-\frac{C-2}{2}\frac{X_{i}(k)+X_{j}(k)}{(2C+\delta)k}\right) = Y_{ij} (\infty) \quad \text{a.s. if } \delta =0,\\
(2) &\lim_{n\to \infty} \frac{1}{n^{2\gamma-1}/(2\gamma-1)} \sum_{k=j}^{n} \frac{Y_{ij}(k)}{k^{2}}  \left(1-\frac{C-2}{2}\frac{X_{i}(k)+X_{j}(k)}{(2C+\delta)k}\right) = Y_{ij} (\infty) \quad \text{a.s. if } \delta <0
\end{align*}
where $Y_{ij}(\infty)$ is as defined in Proposition \ref{prop:degdeg}.

\end{Lemma}

\begin{proof}
Define \[Y_{ij}^{*}(n):=  {Y_{ij}(n)}\left(1-\frac{C-2}{2}\frac{X_{i}(n)+X_{j}(n)}{(2C+\delta)n}\right). \]
Note that for $0<\gamma<1$,
\begin{align*}
\frac{Y_{ij}^{*}(n)}{n^{2\gamma}} &=  \frac{Y_{ij}(n)}{n^{2\gamma}}\left(1-\frac{C-2}{2}\frac{X_{i}(n)+X_{j}(n)}{(2C+\delta)n}\right). \\
                                                    & =  \frac{Y_{ij}(n)}{n^{2\gamma}} \left(1-\frac{C-2}{2}\frac{X_{i}(n)+X_{j}(n)}{n^{\gamma}} \frac{1}{(2C+\delta)n^{1-\gamma}}\right)\\
                                                    & \to Y_{ij}(\infty) \left(1-\frac{C-2}{2}(D_{i}(\infty)+D_{j}(\infty))\times 0\right) = Y_{ij}(\infty) \quad \text{(as $n\to \infty$)}
\end{align*}
which holds a.s. using Propositions \ref{prop:degree} and \ref{prop:degdeg}. Now we can proceed to prove the statements using the same arguments as in Lemma \ref{lem:sumright} by replacing $Y_{ij}$ with $Y_{ij}^{*}$.
\end{proof}

With the aid of all the results above we are in a position to prove Theorem  \ref{thm:growingfriends}.
\begin{proof}[Proof of Theorem \ref{thm:growingfriends}]
Using \eqref{def:nijrel} recursively we have
\begin{align}\label{def:nijrecurrence}
N_{ij}(n+1) = N_{{ij}} (j) + \sum_{k=j+1}^{n+1}\bone_{B_{ij}(n+1)}.
\end{align}
 For any $1\le i<j\le n$ with $C\ge2$,
 
 \begin{align*}
 \E[N_{ij}(n+1)| \PA_{n}^{\delta,C}]  & = N_{ij} (n) + \E[N_{ij}(n+1)-N_{ij}(n)| \PA_{n}^{\delta,C}] \\
                                                      & = N_{ij}(n) + \left[1- \left(1-\frac{D_{i}(n)   + \delta}{(2C+\delta)n}\right)^{C}\right.  \\
                                                      &  \quad \quad-  \left.\left(1-\frac{D_{j}(n)+\delta}{(2C+\delta)n}\right)^{C}+ \left(1-\frac{D_{i}(n)+D_{j}(n)+2\delta}{(2C+\delta)n}\right)^{C}\right].
\end{align*}
We can check that,
\begin{align}
 \E[N_{ij}(n+1)| \PA_{n}^{\delta,C}]  & \le N_{ij}(n) + C(C-1)\frac{Y_{ij}(n)}{(2C+\delta)^{2}n^{2}}, \label{eq:right}
\end{align}
 holding with equality for $C=2$ and 
\begin{align}
  & \E[N_{ij}(n+1)| \PA_{n}^{\delta,C}]  \nonumber\\
  & \ge N_{ij}(n) + C(C-1)\frac{Y_{ij}(n)}{(2C+\delta)^{2}n^{2}}  \left[1- \frac{C-2}{2}\left(\frac{X_{i}(n)+X_{j}(n)}{(2C+\delta)n}\right)\right].\label{eq:left}
\end{align}

\noindent \emph{Proof of part (1).} First we prove the case when $\delta>0$. Clearly as $n\to \infty$, $N_{ij}(n+1) \uparrow N_{ij}(\infty)$ where
\[N_{ij}(\infty) := N_{{ij}} (j) + \sum_{k=j+1}^{\infty}\bone_{B_{ij}(n+1)}.\]
We want to show that $N_{ij}(\infty)<\infty$ a.s.. Taking expectations in \eqref{eq:right} and using \eqref{eq:yijnp1} we get

\begin{align*}
\E[N_{ij}(n+1)] & \le \E[N_{ij}(n)] + \E[Y_{ij}(n)] \frac{C(C-1)}{(2C+\delta)^{2}n^{2}}\\ 
                        & = \E[N_{ij}(n)] + \frac{C(C-1)C_{ij}}{(2C+\delta)^{2}}\frac{\Gamma(n+\gamma_{1})\Gamma(n+\gamma_{2})}{(\Gamma(n+1))^{2}}.
\end{align*}
Applying this argument recursively we get
\begin{align*}
\E[N_{ij}(n+1)]     & \le  \E[N_{ij}(j)] + \frac{C(C-1)C_{ij}}{(2C+\delta)^{2}}\sum_{k=j}^{n}\frac{\Gamma(k+\gamma_{1})\Gamma(k+\gamma_{2})}{(\Gamma(k+1))^{2}}\\
                            & < \E[N_{ij}(j)] + \tilde{C} \sum_{k=j}^{\infty}\frac{\Gamma(k+\gamma_{1})\Gamma(k+\gamma_{2})}{(\Gamma(k+1))^{2}}\\
                           & \le \E[N_{ij}(j)] +\tilde{C} \sum_{k=j}^{\infty} \frac{1}{k^{2-(\gamma_{1}+\gamma_{2})}} \left(1+O\left(\frac1k\right)\right).\end{align*}
where $\tilde{C}= \frac{C(C-1)C_{ij}}{(2C+\delta)^{2}} $. Therefore for any $n$ we have
\begin{align}\label{ineq:pb}
 \sum_{k=j+1}^{n+1}\P({B_{ij}(k)}) = \E[N_{ij}(n+1)]   -  \E[N_{ij}(j)] < \tilde{C} \sum_{k=j}^{\infty} \frac{1}{k^{2-2\gamma}} \left(1+O\left(\frac1k\right)\right) <\infty
\end{align}
since $2-(\gamma_{1}+\gamma_{2})=2-2\gamma = 1+\frac{\delta}{2C+\delta} >1$ for $\delta>0$. Since the right hand side in \eqref{ineq:pb} does not depend on $n$, $ \sum_{k=j+1}^{\infty}\P({B_{ij}(k)}) <\infty$. Using  Borel-Cantelli Lemma \citep[Theorem 2.3.1]{durrett:2019} this implies $$\P(B_{ij} \text{ happens infinitely often})=0$$ and hence $ \sum_{k=j+1}^{\infty}\bone_{B_{ij}(n+1)}<\infty$ a.s. and since $N_{{ij}} (j)\le \max(j-2,C)$ we have $$N_{ij}(\infty)=N_{ij}(j) +\sum_{k=j+1}^{\infty}\bone_{B_{ij}(k)}<\infty \quad \text{a.s.}$$

\noindent \emph{Proof of part (2).} Here we address the case where $\delta=0$. Define $$Q_{n}= \P(B_{ij}(n)|\PA_{n}^{\delta,C}) = \E\left[\bone_{B_{ij}(n)}\big|\PA_{n}^{\delta,C}\right]=  \E[N_{ij}(n)-N_{ij}(n)| \PA_{n}^{\delta,C}] .$$ Using the conditional Borel-Cantelli Lemma \cite[Theorem 4.4.5]{durrett:2019}  we have
\begin{align}\label{con:BCQk}
\frac{\sum_{k=j+1}^{n} \bone_{B_{ij}(k)}}{\sum_{k=j+1}^{n} Q_{k}} \to 1 \quad \text{a.s. on $\Big\{\sum_{k=j+1}^{n} Q_{k}=\infty\Big\}$}
\end{align}
Note that, using \eqref{eq:right} and \eqref{eq:left}, we have for $i<j<n$,
\begin{align*}
 L_{n} &:= C(C-1)\frac{Y_{ij}(n)}{(2C+\delta)^{2}n^{2}}  \left[1- \frac{C-2}{2}\left(\frac{X_{i}(n)+X_{j}(n)}{(2C+\delta)n}\right)\right] \le Q_{n}, \; \text{and, }\\
 R_{n}& := C(C-1)\frac{Y_{ij}(n)}{(2C+\delta)^{2}n^{2}} \ge Q_{n}.
\end{align*}
Using the above recursively we obtain
\begin{align*}
\sum_{k=j+1}^{n} L_{k} \le \sum_{k=j+1}^{n} Q_{k} \le \sum_{k=j+1}^{n} R_{k}.
\end{align*}
Now, from Lemmas \ref{lem:sumright}(1) and \ref{lem:sumleft}(1) we have
\[ \lim_{n\to\infty} \frac{\sum_{k=j+1}^{n} L_{k} }{\log(n)} = \lim_{n\to\infty} \frac{\sum_{k=j+1}^{n} R_{k} }{\log(n)} = \frac{C(C-1)}{(2C+\delta)^{2}} Y_{ij}(\infty) \quad \text{a.s.}\]
Therefore we get

\begin{align}\label{con:Qk}
 \lim_{n\to\infty} \frac{\sum_{k=j+1}^{n} Q_{k} }{\log(n)} = \frac{C(C-1)}{(2C+\delta)^{2}} Y_{ij}(\infty) \quad \text{and } \quad  \Big\{\sum_{k=j+1}^{n} Q_{k} =\infty\Big\} \quad \text{a.s.}
 \end{align}
 Hence from \eqref{con:BCQk} and \eqref{con:Qk} we have
 \begin{align*}
 \lim_{n\to\infty} \frac{N_{ij}(n)}{\log(n)} & = \lim_{n\to\infty} \frac{N_{ij}(j)}{\log(n)} +  \lim_{n\to\infty} \frac{\sum_{k=j+1}^{n} \bone_{B_{ij}(k)}}{\log(n)} \\
                               &= 0 +  \lim_{n\to\infty} \frac{\sum_{k=j+1}^{n} \bone_{B_{ij}(k)}}{\sum_{k=j+1}^{n} Q_{k}}   \frac{\sum_{k=j+1}^{n} Q_{k} }{\log(n)} \\ & =\frac{C(C-1)}{(2C+\delta)^{2}} Y_{ij}(\infty) = \frac{C(C-1)}{(2C+\delta)^{2}}  D_{i}(\infty) D_{j}(\infty) \quad \text{a.s.}
 \end{align*}

 \noindent \emph{Proof of part (3).}  The case where $\delta<0$  can be shown using the same technique as for $\delta=0$ by using Lemmas \ref{lem:sumright}(2) and \ref{lem:sumleft}(2)  in place of Lemmas \ref{lem:sumright}(1) and \ref{lem:sumleft}(1).
\end{proof}

\section{Conclusion}\label{sec:conclusion}
In this paper we establish the rate of growth of the number of common friends for two fixed nodes in a linear preferential attachment model. The growth rate is shown to be static, logarithmic or power-law type depending on the choice of the parameter- $\delta >0, \delta =0$ or $\delta <0$ respectively. We use this result to  prove consistency of an estimator of the number of common friends that is less expensive to compute. Such results will be applicable in both link prediction problems for large dynamic networks as well as detection methods for a preferential attachment model.

This is the first step in showing a more general result regarding the growth behavior for common friends of any randomly chosen pair of nodes and obtaining uniform convergence bounds for estimators of common friends. Further properties of such models and estimation issues are under current investigation.

\section{Acknowledgement}
The authors are very grateful to the referee for insightful comments and also for providing us with precise ideas to fill gaps in parts of the proof of Theorem 2.1.


\bibliographystyle{imsart-nameyear}
\bibliography{bibfilenew}

\end{document}